\documentclass[12pt, reqno]{amsart}
\usepackage{amsmath, amsthm, amscd, amsfonts, amssymb, graphicx, xcolor}
\usepackage{graphicx,amssymb,mathrsfs,amsmath,color,fancyhdr,amsthm}
\usepackage[bookmarksnumbered, colorlinks, plainpages]{hyperref}
\usepackage[all]{xy}
\usepackage{tikz}
\usepackage{pgfplots}
\usepackage{cite}
\usepackage{circuitikz-1.0}
\usepackage{verbatim}
\usepackage{float}
\usepackage{adjustbox}
\usepgfplotslibrary{external}
\usetikzlibrary{shapes, arrows}
\usepackage{color}

\newtheorem{theorem}{Theorem}[section]
\newtheorem{lemma}[theorem]{Lemma}
\newtheorem{proposition}[theorem]{Proposition}
\newtheorem{corollary}[theorem]{Corollary}
\theoremstyle{definition}

\newtheorem{example}[theorem]{Example}

\theoremstyle{remark}
\newtheorem{remark}[theorem]{Remark}
\numberwithin{equation}{section}

\begin{document}
\setcounter{page}{1}

\title[Maximal Rigid Representations of Continuous Quiver]{Maximal Rigid Representations of Continuous Quivers of Type $A$ with Automorphism}

\author[XIAOWEN GAO and MINGHUI ZHAO]{XIAOWEN GAO$^1$ and MINGHUI ZHAO$^2$$^{*}$}

\address{$^{1}$ School of Science, Beijing Forestry University, Beijing 100083, P. R. China}
\email{gaoxiaowen@bjfu.edu.cn (X.Gao)} 

\address{$^{2}$ School of Science, Beijing Forestry University, Beijing 100083, P. R. China}
\email{zhaomh@bjfu.edu.cn (M.Zhao)}

%\dedicatory{This paper is dedicated to Professor ABCD}



\thanks{$^{*}$ Corresponding author}
\date{\today}

\begin{abstract}
Buan and Krause gave a classification of maximal rigid representations for cyclic quivers and counted the number of
isomorphism classes. By using this result, we give a formula on the number of isomorphism classes of a kind of maximal rigid representations for continuous quivers of
type $A$ with automorphism.
\end{abstract} \maketitle

\section{Introduction }
%quiver 的历史，A型的定理 推广到带自同构的A
%刚性表示的历史
%本文解决的问题
%各部分的内容

Quivers play an important role in algebraic representation theory and geometric representation theory. Gabriel gave the definitions of quivers and their representations in \cite{Gabriel1972Unzerlegbare}. He also gave the classification of indecomposable representations of quivers of finite type. 

Representation theory of quivers of type $A$ and its generalizations have important applications in topological data analysis (see \cite{oudot2017persistence}). 
In \cite{Crawley2015Decomposition}, Crawley-Boevey gave a
classification of indecomposable representations of $\mathbb{R}$.
In \cite{2017Interval}, Botnan gave a
classification of indecomposable representations of infinite zigzag.
Igusa, Rock and Todorov introduced general continuous quivers of type
$A$ and classified indecomposable representations in \cite{Igusa2022Continuous}.
In \cite{2020Decomposition}, Hanson and Rock introduced continuous cyclic quivers and classified their indecomposable representations.
In \cite{Appel_2020,Appel_2022,Sala_2019,Sala},\, Appel, Sala
and Schiffmann introduced continuum quivers independently. They also introduced the continuum Kac–Moody algebra and continuum quantum group associated to a continuum quiver.

In \cite{1980Generalizations}, Brenner and Butler  gave the definition of tilting modules. The current definition of tilting modules was given by Happel and Ringel in \cite{happel1982tilted}. In \cite{1982Covering}, Bongartz and Gabriel classified the tilting representations of quivers of type $A$ with linear orientation and counted the number of isomorphism classes. Buan and Krause  gave the classification for cyclic quivers and counted the number of isomorphism classes in \cite{2004Tilting}.
In \cite{liu2023maximalrigidrepresentationscontinuous}, Liu and Zhao gave a formula on the number of isomorphism classes of a kind of maximal rigid representation for continuous quivers of type $A$.

In this paper, the main result is a formula on the number of isomorphism classes of maximal rigid representations of continuous quivers of type $A$ with automorphism. The proof is based on the formula given by Buan and Krause in \cite{2004Tilting} and similar to the proof in \cite{liu2023maximalrigidrepresentationscontinuous}.
As an application, we give a formula on the number of isomorphism classes of a kind of maximal rigid representations of continuous cyclic quivers.

In Section 2, we give some notions and the main result. We recall some results on titling and maximal rigid representations in Section 3. The main result is proved in Section 4. In Section 5, we give some results for continuous cyclic quivers.

\section{Continuous quivers of type A with automorphism}

\subsection{Continuous quivers of type $A$ with automorphism}

Let $\mathbb{R}$ be the set of real numbers and $<$ be the normal order on $\mathbb{R}$.
Following notations in \cite{Igusa2022Continuous}, $A_{\mathbb{R}}=(\mathbb{R},<)$ is called a continuous quiver of type $A$.

Let $k$ be a fixed field. A representation of $A_{\mathbb{R}}$ over $k$ is given by $\mathbf{V}=(\mathbf{V}(x),\mathbf{V}(x,y))$, where $\mathbf{V}(x)$ is a $k$-vector space for any $x\in\mathbb{R}$ and $\mathbf{V}(x,y):\mathbf{V}(x)\rightarrow \mathbf{V}(y)$ is a $k$-linear map for any $x<y\in\mathbb{R}$ satisfying that $\mathbf{V}(x,y)\circ\mathbf{V}(y,z)=\mathbf{V}(x,z)$ for any $x<y<z\in\mathbb{R}$. 
Let $\mathbf{V}=(\mathbf{V}(x),\mathbf{V}(x,y))$ and  $\mathbf{W}=(\mathbf{W}(x),\mathbf{W}(x,y))$ be two representations of $A_{\mathbb{R}}$. A family of $k$-linear maps $\varphi=(\varphi_{x})_{x\in \mathbb{R}}$ is called a morphism from $\mathbf{V}$ to $\mathbf{W}$, if $\varphi_{y}\mathbf{V}(x,y)$= $\mathbf{W}(x,y) \varphi_{x}$, for any $x<y\in \mathbb{R}$.
Denoted by $\mathrm{Rep}_{k}(A_{\mathbb{R}})$ the category of representations of $A_{\mathbb{R}}$.
%and $\mathrm{Rep}_{k}^{pwf}(A_{\mathbb{R}})$ the subcategory of pointwise finite-dimensinal representations of $A_{\mathbb{R}}$.

Consider a map $\sigma:\mathbb{R}\rightarrow\mathbb{R}$ sending $x$ to $x+1$. In this paper, $\mathbf{Q}=(A_{\mathbb{R}},\sigma)$ is called a continuous quiver of type $A$ with automorphism $\sigma$.
A representation of $\mathbf{Q}$ over $k$ is given by a representation $\mathbf{V}=(\mathbf{V}(x),\mathbf{V}(x,y))$ of $A_{\mathbb{R}}$ such that
$\mathbf{V}(x)=\mathbf{V}(\sigma(x))$ and $\mathbf{V}(\sigma(x),\sigma(y))=\mathbf{V}(x,y)$.  
%Let $\mathbf{V}=(\mathbf{V}(x),\mathbf{V}(x,y))$ and  $\mathbf{W}=(\mathbf{W}(x),\mathbf{W}(x,y))$ be two representations of $\mathbf{Q}$. A family of $k$-linear maps $\varphi=(\varphi_{x})_{x\in \mathbb{R}}$ is called a morphism from $\mathbf{V}$ to $\mathbf{W}$, if $\varphi_{y} \mathbf{V}(x,y)$= $\mathbf{W}(x,y) \varphi_{x}$, for any $x<y\in \mathbb{R}$.
Denoted by $\mathrm{Rep}_{k}(\mathbf{Q})$ the category of representations of $\mathbf{Q}$.
%and $\mathrm{Rep}_{k}^{pwf}(\mathbf{Q})$ the subcategory of pointwise finite-dimensinal representations of $\mathbf{Q}$.

%Note that an object $\mathbf{V}$ of $\mathrm{Rep}_{k}(\mathbf{Q})$ can be viewed as an object in $\mathrm{Rep}_{k}(A_{\mathbb{R}})$.

For a representation $\mathbf{V}=(\mathbf{V}(x),\mathbf{V}(x,y))$ of $A_{\mathbb{R}}$, consider a representation $\sigma^\ast(\mathbf{V})=(\sigma^\ast(\mathbf{V})(x),\sigma^\ast(\mathbf{V})(x,y))$, where $\sigma^\ast(\mathbf{V})(x)=\mathbf{V}(\sigma(x))$ and $\sigma^\ast(\mathbf{V})(x,y)=\mathbf{V}(\sigma(x),\sigma(y))$. Note that $\bigoplus_{k\in\mathbb{Z}}(\sigma^{\ast})^k\mathbf{V}$ is a representation of $\mathbf{Q}$.

For  $a\leq b\in\mathbb{R}$, we use the notation $|a,b|$  for one of $(a,b),[a,b),(a,b]$ and $[a,b]$. In this paper, $a=-\infty$ and $b=+\infty$ are allowed. Hence, the notation $|a,b|$ may mean $(a,+\infty)$, $[a,+\infty)$, $(-\infty,b)$, $(-\infty,b]$ or $(-\infty,+\infty)$, too.

Denote $\mathbf{T}'_{|a,b|}=(\mathbf{T}'_{|a,b|}(x),\mathbf{T}'_{|a,b|}(x,y))$ as the following representation $A_{\mathbb{R}}$, where
\begin{equation*}
   \mathbf{T}'_{|a,b|}(x)=\left\{  
\begin{aligned}
k,  & \qquad  x \in |a,b|; \\
0,& \qquad \textrm{otherwise};\\
\end{aligned}
\right.
\end{equation*}
and
\begin{equation*}
  \mathbf{T}'_{|a,b|}(x,y)=\left\{  
\begin{aligned}
1_k,  & \qquad  x<y  \textrm{ and } x,y \in |a,b|; \\
0,& \qquad \textrm{otherwise}. \\
\end{aligned}
\right.
\end{equation*}

%According to \cite{2018Decomposition,Igusa2022Continuous}, the pointwise finite-dimensinal representation of continuous quiver of type $A$ with automorphism can be decomposed as the direct sum of interval representations. 

Let $\mathbf{T}_{|a,b|}=\bigoplus_{k\in\mathbb{Z}}(\sigma^{*})^{k}(\mathbf{T}'_{|a,b|})$. Since the representation $\mathbf{T}'_{|a,b|}$ is indecomposable, so is $\mathbf{T}_{|a,b|}$.
Denote $\mathrm{Rep}'_{k}(\mathbf{Q})$ the subcategory of $\mathrm{Rep}_{k}(\mathbf{Q})$ consisting of representations $\mathbf{M}$ with decomposition
\begin{equation*}
   \mathbf{M}=\bigoplus_{i\in I}\mathbf{T}_{|a_i,b_i|}.
\end{equation*}

\subsection{Maximal rigid representations}

Let $\mathbf{M}$ be an object of category $\mathrm{Rep}'_{k}(\mathbf{Q})$ with decomposition
\begin{equation*}
    \mathbf{M}=\bigoplus_{i\in I}\mathbf{M}_i,
\end{equation*}
where $\mathbf{M}_i$ is indecomposable for any $i\in I$.
The representation $\mathbf{M}$ is called basic, if it satisfies $\mathbf{M}_i \ncong \mathbf{M}_{j}$ for any $i\neq j$. And if it satisfies $\mathrm{Ext}^{1}(\mathbf{M} ,\mathbf{M})=0$, it is called rigid. The representation $\mathbf{M}$ is called maximal rigid, if it is basic and rigid and $\mathbf{M} \oplus \mathbf{N}$ is rigid implies that $\mathbf{N}\in\mathrm{add}\mathbf{M} $ for any object $\textbf{N}$ of category $\mathrm{Rep}'_{k}(\mathbf{Q})$. Similarly, we can define the concept of maximal rigid representations in $\mathrm{Rep}'_{k}(A_\mathbb{R})$.

%According to Lemma 3.1 in \cite{2004Tilting}, we have an lemma of compatible similarly.

\begin{lemma}\label{lemma-rigid}
     Let $\mathbf{M}$ be an object of category $\mathrm{Rep}'_{k}(\mathbf{Q})$. Then the represntation $\mathbf{M}$ is rigid in $\mathrm{Rep}'_{k}(\mathbf{Q})$ if and only if it is rigid in $\mathrm{Rep}'_{k}(A_\mathbb{R})$.
\end{lemma}

\begin{proof}
Assume that $\mathbf{M}$ is not rigid $\mathrm{Rep}'_{k}(A_\mathbb{R})$. Hence there exists an object $\mathbf{N}'$ in $\mathrm{Rep}'_{k}(A_\mathbb{R})$ such that
\begin{equation}\label{exact1}
  0\rightarrow\mathbf{M}\rightarrow\mathbf{N}'\rightarrow \mathbf{M}\rightarrow0  
\end{equation} is exact and non-split.
Since $\mathbf{M}$ is an object in $\mathrm{Rep}'_{k}(\mathbf{Q})$, we have $\mathbf{M}(x)=\mathbf{M}(x+1)$ for any $x\in\mathbb{R}$. Note that $\mathbf{N}'(x)\cong\mathbf{M}(x)\oplus\mathbf{M}(x)$. Hence, we can choose  $\mathbf{N}'$ such that $\mathbf{N}'(x)=\mathbf{N}'(x+1)$ for any $x\in\mathbb{R}$.

The exact sequence \ref{exact1} implies that
\begin{equation}\label{exact2}
0\rightarrow\mathbf{M}|_{[a,a+1]}\rightarrow\mathbf{N}'|_{[a,a+1]}\rightarrow \mathbf{M}|_{[a,a+1]}\rightarrow0\end{equation}
is exact and non-split for some $a\in\mathbb{R}$.
Let $\mathbf{N}$ be the unique object in $\mathrm{Rep}'_{k}(\mathbf{Q})$ such that
\begin{equation*}
   \mathbf{N}(x)=\mathbf{N}'(x)
\end{equation*}
for any $x\in\mathbb{R}$,
and
\begin{equation*}
  \mathbf{N}(x,y)=\mathbf{N}'(x,y)  
\end{equation*}
for $x<y\in[a,a+1]$.
The exact sequence \ref{exact2} implies that
\begin{equation*}
0\rightarrow\mathbf{M}\rightarrow\mathbf{N}\rightarrow \mathbf{M}\rightarrow0\end{equation*}
is exact and non-split. Hence, $\mathbf{M}$ is not rigid $\mathrm{Rep}'_{k}(\mathbf{Q})$.

The converse is direct and we get the desired result.
\begin{comment}
Let $\textbf{M}$ be an object of $\mathrm{Rep}'_{k}(\textbf{Q})$. So we have  $\textbf{M}=\bigoplus_{i\in I} \textbf{T}_{|x_{i},y_{i}|}$. Let $\mathbf{T}_{|a,b|}=\bigoplus_{k\in\mathbb{Z}}(\sigma^{*})^{k}(\mathbf{T}'_{|a,b|})$. We can obtain $\mathbf{M}=\bigoplus_{i\in I}\mathbf{T}'_{|x_i,y_i|}$. $\mathbf{T}'_{|a,b|}$ is representation of $A_\mathbb{R}$. Then, if $\mathrm{Ext}^{1}(\mathbf{T}'_{|x_i,y_i|},\mathbf{T}'_{|x_j,y_j|})=0$ for any $i,j\in I$, we have $\mathrm{Ext}^{1}(\mathbf{M},\mathbf{M})=0$. 

And if $\mathrm{Ext}^{1}(\mathbf{T}'_{|x_i,y_i|},\mathbf{T}'_{|x_j,y_j|})\neq 0$, we have $\mathrm{Ext}^{1}(\mathbf{M},\mathbf{M})\neq 0$. So the represntation $\mathbf{M}$ is rigid in $\mathrm{Rep}'_{k}(\mathbf{Q})$ if and only if it is rigid in $\mathrm{Rep}'_{k}(A_\mathbb{R})$.  
\end{comment}    
\end{proof}

\begin{proposition}
    Let $\mathbf{M}$ be an object of category $\mathrm{Rep}'_{k}(\mathbf{Q})$ and  $\mathbf{M}=\bigoplus_{i\in I}\mathbf{T}'_{|x_i,y_i|}$ in $\mathrm{Rep}'_{k}(A_\mathbb{R})$. The representation $\mathbf{M}$ is rigid if and only if the intervals $|x_i,y_i|$ and $|x_j,y_j|$ are compatible for any $i\neq j$, that is, one of the following conditions are satisfied
    \begin{enumerate}
        \item $| x_i,y_i | \subseteq | x_j,y_j |$ or $| x_i,y_i | \supseteq | x_j,y_j |$;
        \item $y_i<x_j$ or $y_j<x_i$;
        \item $y_i=x_j$, $| x_i,y_i |=| x_i,y_i)$ and $| x_j,y_j |=(x_j,y_j|$;
        \item  $y_j=x_i$, $| x_j,y_j |=| x_j,y_j)$ and $| x_i,y_i |=(x_i,y_i|$.
    \end{enumerate}
\end{proposition}

\begin{proof}
By Lemma 2.2 in \cite{liu2023maximalrigidrepresentationscontinuous} and Lemma \ref{lemma-rigid}, we get this result.
\end{proof}

\subsection{Representations of type $\alpha$}

Let $n$ be an integer greater than $1$ and $\alpha=(a_{i})_{i\in\mathbb{Z}}$ be a sequence such that $a_{0}=0$, $a_{i}<a_{i+1}$ and $a_{i+n}=\sigma(a_{i})$ for any $i\in\mathbb{Z}$.

For an object $\mathbf{M}$ of category $\mathrm{Rep}'_{k}(\mathbf{Q})$ with decomposition $\mathbf{M}=\bigoplus_{i\in I}\mathbf{T}'_{|x_i,y_i|}$ in $\mathrm{Rep}'_{k}(A_\mathbb{R})$,
let $\hat{I}=\{ | x_i,y_i || i\in I\}$.
For any $c\in(a_{s},a_{s+1})$, let
\begin{equation*}
\begin{aligned}
    \phi^{r}_{1}(c,\mathbf{M})&=\{d\,|\,a_{s+1}\leq d, [c,d]\in \hat{I}\};\\
     \phi^{r}_{2}(c,\mathbf{M})&=\{d\,|\,a_{s+1}\leq d, [c,d)\in \hat{I}\};\\
    \phi^{r}_{3}(c,\mathbf{M})&=\{d\,|\,a_{s+1}\leq d, (c,d]\in \hat{I}\};\\
    \phi^{r}_{4}(c,\mathbf{M})&=\{d\,|\,a_{s+1}\leq d, (c,d)\in \hat{I}\};\\
    \phi^{l}_{1}(c,\mathbf{M})&=\{d\,|\,d \leq a_{s}, [d,c]\in \hat{I}\};\\
     \phi^{l}_{2}(c,\mathbf{M})&=\{d\,|\,d \leq a_{s}, [d,c)\in \hat{I}\};\\
    \phi^{l}_{3}(c,\mathbf{M})&=\{d\,|\,d \leq a_{s}, (d,c]\in \hat{I}\};\\
    \phi^{l}_{4}(c,\mathbf{M})&=\{d\,|\,d \leq a_{s}, (d,c)\in \hat{I}\}.\\
\end{aligned}
\end{equation*}

The representation $M$ is called  of type $\alpha$, if it satisfies the following conditions:
\begin{enumerate}
    \item for any $c\not\in \{a_{i}|i\in\mathbb{Z}\}$, we have
    \begin{equation*}
        \begin{aligned}       
            \phi^{r}_{1}(c,\mathbf{M})=\phi^{r}_{3}(c,\mathbf{M})
            \subset\{a_{i}|i\in\mathbb{Z}\},\\
            \phi^{r}_{2}(c,\mathbf{M})=\phi^{r}_{4}(c,\mathbf{M})
            \subset\{a_{i}|i\in\mathbb{Z}\},\\
            \phi^{l}_{1}(c,\mathbf{M})=\phi^{l}_{3}(c,\mathbf{M})
            \subset\{a_{i}|i\in\mathbb{Z}\},\\
            \phi^{l}_{2}(c,\mathbf{M})=\phi^{l}_{4}(c,\mathbf{M})
            \subset\{a_{i}|i\in\mathbb{Z}\};
        \end{aligned}
   \end{equation*}
   and
   \begin{equation*}
    |\phi^{r}_{1}(c,\mathbf{M})|+|\phi^{r}_{2}(c,\mathbf{M})|+|\phi^{l}_{1}(c,\mathbf{M})|+|\phi^{l}_{2}(c,\mathbf{M})|=1.
    \end{equation*}
    \item for any $s\in\mathbb{Z}$ and $c,f\in (a_{s},a_{s+1})$, we have
    \begin{equation*}
        \begin{aligned}
           \phi^{r}_{1}(c,\mathbf{M})=\phi^{r}_{1}(f,\mathbf{M}),\,\,\phi^{r}_{2}(c,\mathbf{M})=\phi^{r}_{2}(f,\mathbf{M}),\\ \phi^{r}_{3}(c,\mathbf{M})=\phi^{r}_{3}(f,\mathbf{M}),\,\,\phi^{r}_{4}(c,\mathbf{M})=\phi^{r}_{4}(f,\mathbf{M}),\\ \phi^{l}_{1}(c,\mathbf{M})=\phi^{l}_{1}(f,\mathbf{M}),\,\,\phi^{l}_{2}(c,\mathbf{M})=\phi^{l}_{2}(f,\mathbf{M}),\\
           \phi^{l}_{3}(c,\mathbf{M})=\phi^{l}_{3}(f,\mathbf{M}),\,\,\phi^{l}_{4}(c,\mathbf{M})=\phi^{l}_{4}(f,\mathbf{M}).
        \end{aligned}
    \end{equation*}
\end{enumerate}

%\begin{remark}
   % In this paper, the interval $(a_{s},a_{s+1})$ is fixed for any $0\leq s\leq n-1$. In this way, we can establish a connection with the representation of type $\mathbf{Q}$. Its settings are similar to that in Section ”Other Dynkin types” of Appendix $B$ in \cite{2019The}.
%\end{remark}

\begin{example}
\label{111}
     Let  $\alpha=(a_{i})_{i\in\mathbb{Z}}$ be a sequence such that $a_{0}=0$, $a_{1}=1$, $a_{i}<a_{i+1}$ and $a_{i+1}=\sigma(a_{i})$ for any $i\in\mathbb{Z}$. Consider the following representations
     \begin{equation*}
         \begin{aligned}
             \textbf{M}_{1}&=\textbf{T}_{(0,1)}\oplus \textbf{T}_{(-\infty,1)}\oplus \bigoplus_{x\in (0,1)} (\textbf{T}_{[x,1)}\oplus \textbf{T}_{(x,1)}),\\
             \textbf{M}_{2}&=\textbf{T}_{(0,1)}\oplus \textbf{T}_{(-\infty,1)}\oplus \bigoplus_{x\in (0,1)} (\textbf{T}_{(0,x)}\oplus \textbf{T}_{(0,x]}),\\
             \textbf{M}_{3}&=\textbf{T}_{[0,0]}\oplus \textbf{T}_{(-\infty,0]}\oplus \bigoplus_{x\in (0,1)} (\textbf{T}_{(-\infty,x)}\oplus \textbf{T}_{(-\infty,x]}),\\
            \textbf{M}_{4}&=\textbf{T}_{[0,0]}\oplus \textbf{T}_{(-\infty,0]}\oplus \bigoplus_{x\in (0,1)} (\textbf{T}_{[x,1]}\oplus \textbf{T}_{(x,1]}),\\
            \textbf{M}_{5}&=\textbf{T}_{(-\infty,0]}\oplus \textbf{T}_{(-\infty,1)}\oplus \bigoplus_{x\in (0,1)} (\textbf{T}_{(-\infty,x)}\oplus \textbf{T}_{(-\infty,x]}),\\
            \textbf{M}_{6}&=\textbf{T}_{(-\infty,0]}\oplus \textbf{T}_{(-\infty,1)}\oplus \bigoplus_{x\in (0,1)} (\textbf{T}_{[x,1)}\oplus \textbf{T}_{(x,1)}).\\
            \textbf{M}_{7}&=\textbf{T}_{(0,1)}\oplus \textbf{T}_{(0,+\infty)}\oplus \bigoplus_{x\in (0,1)} (\textbf{T}_{[x,1)}\oplus \textbf{T}_{(x,1)}),\\
             \textbf{M}_{8}&=\textbf{T}_{(0,1)}\oplus \textbf{T}_{(0,+\infty)}\oplus \bigoplus_{x\in (0,1)} (\textbf{T}_{(0,x)}\oplus \textbf{T}_{(0,x]}),\\
             \textbf{M}_{9}&=\textbf{T}_{[0,0]}\oplus \textbf{T}_{[0,+\infty)}\oplus \bigoplus_{x\in (0,1)} (\textbf{T}_{[0,x)}\oplus \textbf{T}_{[0,x]}),\\
            \textbf{M}_{10}&=\textbf{T}_{[0,0]}\oplus \textbf{T}_{[0,+\infty)}\oplus \bigoplus_{x\in (0,1)} (\textbf{T}_{[x,+\infty)}\oplus \textbf{T}_{(x,+\infty)}),\\
            \textbf{M}_{11}&=\textbf{T}_{[0,+\infty)}\oplus \textbf{T}_{(0,+\infty)}\oplus \bigoplus_{x\in (0,1)} (\textbf{T}_{(0,x)}\oplus \textbf{T}_{(0,x]}),\\
            \textbf{M}_{12}&=\textbf{T}_{[0,+\infty)}\oplus \textbf{T}_{(0,+\infty)}\oplus \bigoplus_{x\in (0,1)} (\textbf{T}_{[x,+\infty)}\oplus \textbf{T}_{(x,+\infty)}).\\
         \end{aligned}
     \end{equation*}
     The representations $\textbf{M}_{i}$ are of type $\alpha$ for $i=1,2,\cdots,12$. And the set 
     \begin{equation*}
         \{[\textbf{M}_{i}]|i=1,2,\cdots,12\}
     \end{equation*}
      gives all the isomorphism classes of maximal rigid representations of type $\alpha$.
\end{example}

The main result in this paper is about the number of isomorphism classes of maximal rigid representation of $\mathbf{Q}$ of type $\alpha$.

\begin{theorem}
\label{main result}
Let $\mathbf{Q}$ be the continuous quiver of type $A$ with automorphism and $\alpha=(a_{i})_{i\in\mathbb{Z}}$ be a sequence such that $a_{0}=0$, $a_{i}<a_{i+1}$ and $a_{i+n}=\sigma(a_{i})$ for any $i\in\mathbb{Z}$. Let 
\begin{equation*}
     \mathbf{A}=\{[\mathbf{M}]|\mathbf{M} \textrm{ is a maximal  rigid representation of $\mathbf{Q}$ of type $\alpha$}\}.
\end{equation*}
Then
\begin{equation*}
        |\mathbf{A}|=2^{n+1}\left(
        \begin{aligned}
            4n-1\\
            2n-1
        \end{aligned}
        \right).
    \end{equation*}
                                                   
\end{theorem}

The proof of this theorem will be given in Section \ref{proof theorem}.

\section{Titling representations for cyclic quivers}

\subsection{Quivers and representations}

Let $Q=(Q_0,Q_1,h,t)$ be a quiver, where $Q_0$ is the set of vertices, $Q_1$ is the set of arrows, $h,t:Q_1\rightarrow Q_0$ are two maps such that $h(\alpha)$ is the head and $t(\alpha)$ is the tail of an arrow $\alpha$.

A representation of $Q$ over $k$ is given by $V=(V(a),V(\rho))$, where
${V}(a)$ is a $k$-vector space for any $a\in{Q}_0$ and ${V}(\rho):{V}(t(\rho))\rightarrow{V}(h(\rho))$ is a $k$-linear map for any $\rho\in{Q}_1$.
Let ${V}=({V}(a),{V}(\rho))$ and  ${W}=({W}(a),{W}(\rho))$ be two representations of ${Q}$. A family of $k$-linear maps $\varphi=(\varphi_a)$ is called a morphism from $V$ to $W$, if $\varphi_{h(\rho)}V(\rho)=W(\rho)\varphi_{t(\rho)}$ for any $\rho\in Q_1$.
Denote by $\mathrm{Rep}_{k}({Q})$ the category of representations of ${Q}$.

A representation $M$ of ${Q}$ over $k$ is called basic if the direct summands of $M$ are not
isomorphic to each other, and is called rigid if $\mathrm{Ext}^{1}(M,M)=0$.
A rigid representation $M$ is called tilting if the projective dimension of $M$ is at most $1$, and there exists a short exact sequence $0\rightarrow P\rightarrow M_0\rightarrow M_1\rightarrow 0$ with each $M_i$ in $\mathrm{add}M$ for any indecomposable projective representation $P$ of ${Q}$.

\subsection{Representations for cyclic quivers}

Let ${\Delta_n}=({Q}_0,{Q}_1,h,t)$ be a cyclic quiver, where ${Q}_0=(0,1,\ldots,n-1)$, ${Q}_1=(\alpha_{1},\alpha_{2},\cdots,\alpha_{n})$, $h(\alpha_{i})=i-1$ for any $i=0,1,\ldots,n-1$, $t(\alpha_{i})=i$ for any $i=0,1,\ldots,n-2$ and $t(\alpha_{n})=0$.

\begin{proposition}[\cite{2004Tilting}]\label{number-tilting}
\label{proposotion1}
    Let $$\tilde{A}=\{[M]|M\textrm{ is a basic and tilting representation of }\Delta_n\},$$ then
    \begin{equation*}
        |\tilde{A}|=\left(
        \begin{aligned}
            2n-1\\
            n-1
        \end{aligned}
        \right).
    \end{equation*}
\end{proposition}

\subsection{Representations of quiver of type $A_{\infty}$ with automorphism}
Let $Q=(Q_0,Q_1,h,t)$ be a quiver, where $Q_0=\mathbb{Z}$, $Q_1=\{\alpha_{i}|i\in\mathbb{Z}\}$, $t(\alpha_{i})=i-1$ and 
    $h(\alpha_{i})=i$.
The quiver $Q$ is of type $A_{\infty}$ and with linear orientation.

Consider two maps $\sigma_n:Q_0\rightarrow Q_0$ sending $i$ to $i+n$
and $\sigma_n:Q_1\rightarrow Q_1$ sending $\alpha_{i}$ to $\alpha_{i+n}$.
A representation of $(Q,\sigma_n)$ over $k$ is a representaiton $V=(V(i),V(\rho))$ of $Q$
such that ${V}(i)={V}(\sigma_n(i))$ and ${V}(\sigma_n(\rho))={V}(\rho)$ for any $i\in Q_0$ and $\rho\in Q_1$.

%Let $V=(V(i),V(\rho))$ and  $W=(W(i),W(\rho))$ be two representations of $Q$. A family of $k$-linear maps $\varphi=(\varphi_{i})_{i\in \mathbb{R}}$ is called a morphism from $\mathbf{V}$ to $\mathbf{W}$, if $\varphi_{y}\mathbf{V}(x,y)$= $\mathbf{W}(x,y) \varphi_{x}$, for any $x<y\in \mathbb{R}$.
%Denoted by $\mathrm{Rep}_{k}(A_{\mathbb{R}})$ the category of representations of $A_{\mathbb{R}}$ and $\mathrm{Rep}_{k}^{pwf}(A_{\mathbb{R}})$ the subcategory of pointwise finite-dimensinal representations of $A_{\mathbb{R}}$.

For a  representation ${V}=({V}(i),{V}(\rho))$ of $Q$, let $\sigma_n^\ast(V)=(\sigma_n^\ast(V)(i),\sigma_n^\ast(V)(\rho))$, where $\sigma_n^\ast(V)(i)=V(\sigma_n(i))$ and $\sigma_n^\ast(V)(\rho)=V(\sigma_{n}(\rho))$. Note that $\bigoplus_{k\in\mathbb{Z}}(\sigma_n^{\ast})^kV$ is a representation of $(Q,\sigma_n)$.
Denoted by $\mathrm{Rep}_{k}(Q,\sigma_n)$ the category of representations of $(Q,\sigma_n)$. 
Note that there exits an equivalence $$\Phi:\mathrm{Rep}_{k}(\Delta_n)\rightarrow \mathrm{Rep}_{k}(Q,\sigma_n)$$ such that $\Phi(V)(i)=V(\hat{i})$ and $\Phi(V)(\alpha_i)=V(\alpha_{\hat{i}})$, where $\hat{i}$ is the remainder of $i$ divided by $n$. The inverse of $\Phi$ is denoted by $\Psi$.

For a  representation ${V}=({V}(i),{V}(\rho))$ of $Q$, let $V^\ast=(V^\ast(i),V^\ast(\rho))$, where $V^\ast(i)=V(-i)$ and $V^\ast(\alpha_i)=V(\alpha_{-i-1})^T$.

For $a,b\in Q_0$ such that $a\leq b$, denoted by $T'_{[a,b]}=(T'_{[a,b]}(i),T'_{[a,b]}(\rho))$ the following representation of $Q$, where
\begin{equation*}
  T'_{[a,b]}(i)=\left\{  
\begin{aligned}
k,  & \qquad  x\in Q_0 \textrm{ and } a\leq x\leq b;\\
0, & \qquad \textrm{otherwise}; \\
\end{aligned}
\right.
\end{equation*}
and
\begin{equation*}
   T'_{[a,b]}(\rho)=\left\{  
\begin{aligned}
1_k,  & \qquad  a\leq t_{\rho}<h_{\rho}\leq b; \\
0,& \qquad \textrm{otherwise}. \\
\end{aligned}
\right.
\end{equation*}
Let ${T}_{[a,b]}=\bigoplus_{k\in\mathbb{Z}}(\sigma_n^{\ast})^{k}(T'_{[a,b]})$, which is a representation of $(Q,\sigma_n)$.

%Similar to continuous quiver, there also exist an equivalence between the representation of $\tilde{Q}$ and $Q$. So, we have the following lemma, which is similar to \ref{representations equal}.

\begin{lemma}
\label{lemma3.3}
For any representation $M$ of $(Q,\sigma_n)$, $M$ is maximal rigid if and only if $\Psi(M)$ or $\Psi(M^\ast)$ is basic and tilting. 
\end{lemma}
\begin{proof}
Similarly to the proof of Lemma 3.2 in \cite{liu2023maximalrigidrepresentationscontinuous}, $\Psi(M)$ is basic and tilting implies that $M$ is maximal rigid. Since $M^\ast$ is maximal rigid if and only if $M$ is maximal rigid,  $\Psi(M^\ast)$ is basic and tilting implies that $M$ is maximal rigid, too.

Conversely, assume that $$M=\bigoplus_{i\in I_1}{T}_{[a_i,b_i]}\oplus\bigoplus_{i\in I_2}{T}_{(-\infty,d_i]}\oplus\bigoplus_{i\in I_3}{T}_{[c_i,+\infty)}.$$
Since $M$ is rigid, we have  $$M=\bigoplus_{i\in I_1}{T}_{[a_i,b_i]}\oplus\bigoplus_{i\in I_2}{T}_{(-\infty,d_i]}$$or$$M=\bigoplus_{i\in I_1}{T}_{[a_i,b_i]}\oplus\bigoplus_{i\in I_3}{T}_{[c_i,+\infty)}.$$
In the first case, the representation $M$ is maximal rigid implies that $\Psi(M)$ is basic and tilting. In the second case, the representation $M$ is maximal rigid implies that $\Psi(M^\ast)$ is basic and tilting.  
\end{proof}

\begin{corollary}
\label{co2}
Let $$A=\{[M]|M\textrm{ is a maximal rigid representation of }(Q,\sigma_n)\},$$ then
 \begin{equation*}
        |A|=2\left(
        \begin{aligned}
            2n-1\\
            n-1
        \end{aligned}
        \right).
    \end{equation*}
\end{corollary}

\begin{proof}
By Lemma \ref{lemma3.3}, there is a map $\Phi:\tilde{A}\rightarrow A$.
Note that a representation $N$ is tilting is quivalent to that $N^\ast$ is not. Hence, $|\Phi^{-1}([M])|=2$ for any maximal rigid representation $M$ of $(Q,\sigma_n)$. By Proposition \ref{number-tilting}, we have
\begin{equation*}
        |A|=2|\tilde{A}|=2\left(
        \begin{aligned}
            2n-1\\
            n-1
        \end{aligned}
        \right).
    \end{equation*}
\end{proof}

\section{Proof of Theorem \ref{main result}}
\label{proof theorem}
In this part, we will give a proof of Theorem \ref{main result}. Let $\mathbf{Q}$ be the continuous quiver of type $A$ with automorphism and $\alpha=(a_{i})_{i\in\mathbb{Z}}$ be a sequence such that $a_{0}=0$, $a_{i}<a_{i+1}$ and $a_{i+n}=\sigma(a_{i})$ for any $i\in\mathbb{Z}$. 
Denote
\begin{equation*}
      \mathbf{A}=\{[\mathbf{M}]|\textrm{$\mathbf{M}$ is a maximal rigid representation  of $\mathbf{Q}$ of type $\alpha$}\}
\end{equation*}
and 
\begin{equation*}
      \mathbf{B}=\{[\mathbf{M}]|\textrm{$\mathbf{M}$ is a rigid representation  of $\mathbf{Q}$ of type $\alpha$}\}.
\end{equation*}

Let $\bar{Q}=(\bar{Q}_0,\bar{Q}_1,\bar{\sigma})$ be the following quiver with automorphism $\bar{\sigma}$,
\\
\begin{center}
    \begin{tikzpicture}
        \filldraw[black] (-1.3,0) circle (1pt);
        \filldraw[black] (-1.5,0) circle (1pt);
        \filldraw[black] (-1.7,0) circle (1pt);
        \filldraw[black] (-1,0) circle (2pt);
        \node(a) at (-1,-0.75) {$a^-_0$};
        \draw[thick, ->] (-0.8,0)--(-0.2,0);
        \filldraw[black] (0,0) circle (2pt);
        \draw[thick, ->] (0.2,0)--(0.8,0);
        \node(a) at (0,-0.8) {$a_0$};
        \filldraw[black] (1,0) circle (2pt);
        \draw[thick, ->] (1.2,0)--(1.8,0);
        \node(a) at (1,-0.75) {$a_{0}^{+}$};
        \filldraw[black] (2,0) circle (2pt);
        \draw[thick, ->] (2.2,0)--(2.8,0);
        \node(a) at (2,-0.75) {$a_{1}^{-}$};
        \filldraw[black] (3,0) circle (2pt);
        \draw[thick, ->] (3.2,0)--(3.8,0);
        \node(a) at (3,-0.8) {$a_1$};
        \filldraw[black] (4,0) circle (2pt);
        \draw[thick, ->] (4.2,0)--(4.8,0);
        \node(a) at (4,-0.75) {$a_{1}^{+}$};
        \filldraw[black] (5.3,0) circle (1pt);
        \filldraw[black] (5.5,0) circle (1pt);
        \filldraw[black] (5.7,0) circle (1pt);
        \filldraw[black] (5.3,-0.75) circle (1pt);
        \filldraw[black] (5.5,-0.75) circle (1pt);
        \filldraw[black] (5.7,-0.75) circle (1pt);
        \draw[thick, ->] (6.2,0)--(6.8,0);
        \filldraw[black] (7,0) circle (2pt);
        \draw[thick, ->] (7.2,0)--(7.8,0);
        \node(a) at (7,-0.75) {$a_{n}^{-}$};
        \filldraw[black] (8,0) circle (2pt);
        \draw[thick, ->] (8.2,0)--(8.8,0);
        \node(a) at (8,-0.8) {$a_n$};
        \filldraw[black] (9,0) circle (2pt);
        %\draw[thick, ->] (9.2,0)--(9.8,0);
        \node(a) at (9,-0.75) {$a_{n}^{+}$};
        %\filldraw[black] (10,0) circle (2pt);
        %\draw[thick, ->] (10.2,0)--(10.8,0);
        %\node(a) at (10,-0.75) {$a^-_{{n+1}}$};
        %\filldraw[black] (11,0) circle (2pt);
        %\node(a) at (11,-0.8) {$a_{{n+1}}$};
        %\draw[thick, ->] (11.2,0)--(11.8,0);
        %\filldraw[black] (12,0) circle (2pt);
        %\node(a) at (12,-0.75) {$a^+_{{n+1}}$};
        \filldraw[black] (9.3,0) circle (1pt);
        \filldraw[black] (9.5,0) circle (1pt);
        \filldraw[black] (9.7,0) circle (1pt);
    \end{tikzpicture}
\end{center}
where $\bar{\sigma}:\bar{Q}_0\rightarrow\bar{Q}_0$ is a map sending $a_i$, $a^{+}_i$ and $a^{-}_i$ to $a_{i+n}$, $a^{+}_{i+n}$ and $a^{-}_{i+n}$, respectively.
Denote
\begin{equation*}
     \begin{aligned}
        \bar{B}=\{[M]|M=\bigoplus_{j\in I}T_{[x_j,y_j]}\textrm{ is a rigid  representation of $\bar{Q}$}\\
        \textrm{such than $x_j\neq a^{-}_{i}$ and $y_j\neq a^{+}_{i}$ for any  $i\in\mathbb{Z}$ and $j\in I$}\}.
      \end{aligned}
\end{equation*}

Let $\hat{Q}=(\hat{Q}_0,\hat{Q}_1,\hat{\sigma})$ be the following quiver with automorphism $\hat{\sigma}$,
\\
\begin{center}
    \begin{tikzpicture}
        \filldraw[black] (-1.3,0) circle (1pt);
        \filldraw[black] (-1.5,0) circle (1pt);
        \filldraw[black] (-1.7,0) circle (1pt);
        \filldraw[black] (-1,0) circle (2pt);
        \node(a) at (-1,-0.8) {$a_{-1,0}$};
        \draw[thick, ->] (-0.8,0)--(-0.2,0);
        \filldraw[black] (0,0) circle (2pt);
        \draw[thick, ->] (0.2,0)--(0.8,0);
        \node(a) at (0,-0.8) {$a_0$};
        \filldraw[black] (1,0) circle (2pt);
        \draw[thick, ->] (1.2,0)--(1.8,0);
        \node(a) at (1,-0.8) {$a_{0,1}$};
        \filldraw[black] (2,0) circle (2pt);
        \draw[thick, ->] (2.2,0)--(2.8,0);
        \node(a) at (2,-0.8) {$a_{1}$};
        \filldraw[black] (3,0) circle (2pt);
        \draw[thick, ->] (3.2,0)--(3.8,0);
        \node(a) at (3,-0.8) {$a_{1,2}$};
        \filldraw[black] (4,0) circle (2pt);
        \draw[thick, ->] (4.2,0)--(4.8,0);
        \node(a) at (4,-0.8) {$a_{2}$};
        \filldraw[black] (5.3,0) circle (1pt);
        \filldraw[black] (5.5,0) circle (1pt);
        \filldraw[black] (5.7,0) circle (1pt);
        \filldraw[black] (5.3,-0.75) circle (1pt);
        \filldraw[black] (5.5,-0.75) circle (1pt);
        \filldraw[black] (5.7,-0.75) circle (1pt);
        \draw[thick, ->] (6.2,0)--(6.8,0);
        \filldraw[black] (7,0) circle (2pt);
        \draw[thick, ->] (7.2,0)--(7.8,0);
        \node(a) at (7,-0.75) {$a_{n-1,n}$};
        \filldraw[black] (8,0) circle (2pt);
        \draw[thick, ->] (8.2,0)--(8.8,0);
        \node(a) at (8,-0.8) {$a_n$};
        \filldraw[black] (9,0) circle (2pt);
        \node(a) at (9,-0.8) {$a_{n,n+1}$};
        \filldraw[black] (9.3,0) circle (1pt);
        \filldraw[black] (9.5,0) circle (1pt);
        \filldraw[black] (9.7,0) circle (1pt);
    \end{tikzpicture}
\end{center}
where $\hat{\sigma}:\hat{Q}_0\rightarrow\hat{Q}_0$ is a map sending $a_i$ and $a_{i,i+1}$ to $a_{i+n}$ and $a_{i+n,i+n+1}$, respectively.
Denote
\begin{equation*}
      \hat{A}=\{[M]|M\,\, is\,\, a\, \,maximal\,\, rigid\,\, representation\,\, of\,\, \hat{Q}\}
\end{equation*}
and 
\begin{equation*}
      \hat{B}=\{[M]|M\,\, is\,\, rigid\,\, representation\,\, of\,\, \hat{Q}\}.
\end{equation*}

%\xymatrix{
%\bullet \ar[r] & \bullet
%}

\begin{comment}
\begin{remark}
    Because $\alpha=(a_{0},a_{1},a_{2},\cdots,a_{n})$ with $a_{0}=0$ and satisfies $a_{i}<a_{i+1}\in [0,1)$ for any $0\leq i\leq n-1$ and $\mathbf{Q}$ is continuous quiver of type $A$ with automorphism, the information of point $a_{{0}^{'}}$ is equivalent to the information of point $a_{0}$.
\end{remark}    
\end{comment}

Consider a map $\tau_1:\mathbf{B}\rightarrow\bar{B}$ defined by 
\begin{enumerate}
    \item for any $x\not\in\{a_{i}|i\in\mathbb{Z}\}$,
    \begin{equation*}
        \tau_{1}([\mathbf{T}_{| a_{i},x|}])=0,\,\tau_{1}([\mathbf{T}_{| x,a_{i}|}])=0;
    \end{equation*}
    \item for any $i<j\in\mathbb{Z}$,
    \begin{equation*}
        \begin{aligned}
            \tau_{1}([\mathbf{T}_{[a_{i},a_{j}]}])&=[T_{[a_{i},a_{j}]}],\\\tau_{1}([\mathbf{T}_{(a_{i},a_{j}]}])&=[T_{[a^{+}_{i},a_{j}]}],\\\tau_{1}([\mathbf{T}_{[a_{i},a_{j})}])&=[T_{[a_{i},a^{-}_{j}]}],\\\tau_{1}([\mathbf{T}_{(a_{i},a_{j})}])&=[T_{[a^{+}_{i},a^{-}_{j}]}];
        \end{aligned}
    \end{equation*}
    %\item for any $0\leq i \leq n$,
    %\begin{equation*}
        %\begin{aligned}
            %\eta_{1}([\mathbf{T}_{[a_{i},a_{{0}^{'}}]}])&=[T_{[a_{i},a_{{0}^{'}}]}],\\ \eta_{1}([\mathbf{T}_{(a_{i},a_{{0}^{'}}]}])&=[T_{[a^{+}_{i},a_{{0}^{'}}]}],\\ \eta_{1}([\mathbf{T}_{[a_{i},a_{{0}^{'}})}])&=[T_{{[a_{i},a_{{0}^{'}}}^{-}]}],\\ \eta_{1}([\mathbf{T}_{(a_{i},a_{{0}^{'}})}])&=[T_{{[a^{+}_{i},a_{{0}^{'}}}^{-}]}];
        %\end{aligned}
    %\end{equation*}
    \item for any $i\in\mathbb{Z}$,
    \begin{equation*}
        \tau_{1}([\mathbf{T}_{[a_{i},a_{i}]}])=[T_{[a_{i},a_{i}]}];
    \end{equation*}
    \item for any $[\mathbf{M}],[\mathbf{N}]\in \mathbf{B}$,
    \begin{equation*}
        \tau_{1}([\mathbf{M}\oplus \mathbf{N}])=[M\oplus N],
    \end{equation*}
    where $[M]=\tau_{1}([\mathbf{M}])$ and $[N]=\tau_{1}([\mathbf{N}])$.
\end{enumerate}

Consider a map $\tau_{2}:\bar{B}\rightarrow \hat{B}$ defined by
\begin{enumerate}
    \item for any $i<j\in\mathbb{Z}$,
    \begin{equation*}
        \begin{aligned}
            \tau_{2}([T_{[a_{i},a_{j}]}])&=[T_{[a_{i},a_{j}]}],\\\tau_{2}([T_{[a^{+}_{i},a_{j}]}])&=[T_{[a_{i,i+1},a_{j}]}],\\\tau_{2}([T_{[a_{i},a^{-}_{j}]}])&=[T_{[a_{i},a_{j-1,j}]}],\\\tau_{2}([T_{[a^{+}_{i},a^{-}_{j}]}])&=[T_{[a_{i,i+1},a_{j-1,j}]}];
        \end{aligned}
    \end{equation*}
    \item for any $i\in\mathbb{Z}$,
    \begin{equation*}
        \tau_{2}([T_{[a_{i},a_{i}]}])=[T_{[a_{i},a_{i}]}];
    \end{equation*}
    \item for any $[M],[N]\in \bar{B}$,
    \begin{equation*}
        \tau_{2}([M\oplus N])=[\hat{M}\oplus \hat{N}],
    \end{equation*}
    where $[\hat{M}]=\tau_{2}([M])$ and $[\hat{N}]=\tau_{2}([N])$.
\end{enumerate}

Similarly to Lemma 4.1 in \cite{liu2023maximalrigidrepresentationscontinuous}, we have the following lemma.

\begin{lemma}\label{lemma2}
    With above notations, we have
    \begin{enumerate}
        \item the map $\tau=\tau_{2}\tau_{1}:\mathbf{B}\rightarrow\hat{B}$ is surjective and $\tau(\mathbf{A})=\hat{A}$;
        \item for any $\hat{M}\in \hat{A}$, we have $|\tau_{\mathbf{A}}^{-1}([\hat{M}])|=2^{n}$, where $\tau_{\mathbf{A}}$ is the restriction of $\tau$ on $\mathbf{A}$.
    \end{enumerate}
\end{lemma}

The following example is an illustration of Lemma \ref{lemma2}.
\begin{example}
    Let  $\alpha=(a_{i})_{i\in\mathbb{Z}}$ be a sequence such that $a_{0}=0$, $a_{1}=1$, $a_{i}<a_{i+1}$ and $a_{i+1}=\sigma(a_{i})$ for any $i\in\mathbb{Z}$. In Example \ref{111}, we have 
    \begin{equation*}
        \mathbf{A}=\{[\mathbf{M}_{i}]|i=1,2,\cdots,12\}.
    \end{equation*}
    Now, $\hat{Q}$ is the following quiver
    \begin{center}
    \begin{tikzpicture}
        \filldraw[black] (-1.3,0) circle (1pt);
        \filldraw[black] (-1.5,0) circle (1pt);
        \filldraw[black] (-1.7,0) circle (1pt);
        \filldraw[black] (-1,0) circle (2pt);
        \node(a) at (-1,-0.8) {$a_{-1,0}$};
        \draw[thick, ->] (-0.8,0)--(-0.2,0);
        \filldraw[black] (0,0) circle (2pt);
        \draw[thick, ->] (0.2,0)--(0.8,0);
        \node(a) at (0,-0.8) {$a_0$};
        \filldraw[black] (1,0) circle (2pt);
        \draw[thick, ->] (1.2,0)--(1.8,0);
        \node(a) at (1,-0.8) {$a_{0,1}$};
        \filldraw[black] (2,0) circle (2pt);
        \node(a) at (2,-0.8) {$a_{1}$};
        \filldraw[black] (3,0) circle (2pt);
        \node(a) at (3,-0.8) {$a_{1,2}$};
         \draw[thick, ->] (2.2,0)--(2.8,0);
        \filldraw[black] (3.3,0) circle (1pt);
        \filldraw[black] (3.5,0) circle (1pt);
        \filldraw[black] (3.7,0) circle (1pt);
        %\node(a) at (2,0.5) {$1$};
        %\node(a) at (0,0.5) {$0$};
    \end{tikzpicture}
    \end{center}
     and
    \begin{equation*}
        \hat{A}=\{[M_{i}]|i=1,2,\ldots,6\},
    \end{equation*}
    where
    \begin{equation*}
        \begin{aligned}
            M_{1}&=T_{[a_{0,1},a_{0,1}]}\oplus T_{(-\infty,a_{0,1}]},\\
            M_{2}&=T_{[a_{0},a_{0}]}\oplus T_{(-\infty,a_{0}]},\\
            M_{3}&=T_{(-\infty,a_{0,1}]}\oplus T_{(-\infty,a_{0}]},\\
            M_{4}&=T_{[a_{0,1},a_{0,1}]}\oplus T_{[a_{0,1},+\infty)},\\
            M_{5}&=T_{[a_{0},a_{0}]}\oplus T_{[a_{0},+\infty)},\\
            M_{6}&=T_{[a_{0,1},+\infty)}\oplus T_{[a_{0},+\infty)}.
        \end{aligned}
    \end{equation*}
Now, $\tau(\mathbf{M}_{2k})=\tau(\mathbf{M}_{2k-1})=M_{k}$ and $|\tau^{-1}_{\mathbf{A}}([M_{k}])|=2$ for $k=1,2,\ldots,6$.
\end{example}

Based on above discussion, we give a proof of Theorem \ref{main result}.

\begin{proof}
According to Corollary \ref{co2}, we have
    \begin{equation*}
        |\hat{A}|=2\left(
        \begin{aligned}
            4n-1\\
            2n-1
        \end{aligned}
        \right)
    \end{equation*}
    Then, we have
    \begin{equation*}
        |\mathbf{A}|=2^{n}|\hat{A}|=2^{n+1}\left(
        \begin{aligned}
            4n-1\\
            2n-1
        \end{aligned}
        \right)
    \end{equation*}
by Lemma \ref{lemma2}.
\end{proof}

\section{Continuous quivers of type $\tilde{A}$}

%In this section, we give an extension of the result with automorphism based on the relationship between continuous quiver of type $A$ with automorphism and continuous quiver of type $\tilde{A}$, namely the number of isomorphism classes of maximal rigid representations of continuous quiver of type $\tilde{A}$.

Consider an equivalence relationship "${\sim}$" on $\mathbb{R}$, where $x\sim y$ if and only if $y-x\in\mathbb{Z}$. Let $[x]$ be the equivalent class of $x$ for any $x\in\mathbb{R}$ and define $[x]<[y]$ if and only if $x<y$ and $|x-y|<\frac{1}{2}$. Let $\tilde{\mathbf{A}}_{\mathbb{R}}=(\mathbb{R}/{\sim},<)$, which is called a continuous quiver of type $\tilde{A}$. We use the following diagram to represent this quiver.

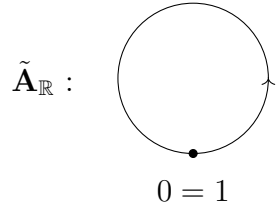
\begin{center}
\begin{tikzpicture}
  \node(c) at (-2,0) {$\tilde{\mathbf{A}}_{\mathbb{R}}$ :};
  \node(d) at (0,-1.5) {$0=1$};
  \draw (0,0) circle (1cm);
  \draw[->] (1,0)--(1,0);
  \filldraw[black] (0,-1) circle (1.5pt);
\end{tikzpicture}
\end{center}

A representation of $\tilde{\mathbf{A}}_{\mathbb{R}}$ over $k$ is given by $\tilde{\mathbf{V}}=(\tilde{\mathbf{V}}([x]),\tilde{\mathbf{V}}([x],[y]))$, where $\tilde{\mathbf{V}}([x])$ is a $k$-vector space, and $\tilde{\mathbf{V}}([x],[y]):\tilde{\mathbf{V}}([x])\rightarrow \tilde{\mathbf{V}}([y])$ is a $k$-linear map for any $[x]<[y]\in \mathbb{R}/{\sim}$ satisfying that $\mathbf{V}([x],[y])\circ\mathbf{V}([y],[z])=\mathbf{V}([x],[z])$ for any $[x]<[y]<[z]\in\mathbb{R}/{\sim}$.
Let $\tilde{\mathbf{V}}=(\tilde{\mathbf{V}}([x]),\tilde{\mathbf{V}}([x],[y]))$ and  $\tilde{\mathbf{W}}=(\tilde{\mathbf{W}}([x]),\tilde{\mathbf{W}}([x],[y]))$ be representations of $\tilde{\mathbf{A}}_{\mathbb{R}}$. 
A family of $k$-linear maps $\varphi_{[x]}:\tilde{\mathbf{V}}([x])\rightarrow \tilde{\mathbf{W}}([x])$ is called a morphism from $\tilde{\mathbf{V}}$ to $\tilde{\mathbf{W}}$, if it satisfies $\varphi_{[y]} \tilde{\mathbf{V}}([x],[y])$= $\tilde{\mathbf{W}}([x],[y]) \varphi_{[x]}$ for any $[x]<[y]$.

Denoted by $\mathrm{Rep}_{k}(\tilde{\mathbf{A}}_{\mathbb{R}})$ the category of representations of $\tilde{\mathbf{A}}_{\mathbb{R}}$.
Note that there exists an equivalence $$\Theta:\mathrm{Rep}_{k}(\tilde{\mathbf{A}}_{\mathbb{R}})\rightarrow \mathrm{Rep}_{k}(\mathbf{Q})$$
such that
  $\Theta(\tilde{\mathbf{V}})(x)=\tilde{\mathbf{V}}([x])$ for any $x\in\mathbb{R}$ and 
   $\Theta(\tilde{\mathbf{V}})(x,y)=\tilde{\mathbf{V}}([x],[y])$
   for any $x<y\in\mathbb{R}$ such that $|x-y|<\frac{1}{2}$.

For any $a\leq b\in\mathbb{R}$, consider the representation $\tilde{\mathbf{T}}_{|a,b|}$ of $\tilde{\mathbf{A}}_{\mathbb{R}}$,
where
\begin{equation*}
\tilde{\mathbf{T}}_{|a,b|}([x])=\bigoplus_{x\in[x]}\mathbf{T}'_{|a,b|}(x),
\end{equation*}
\begin{equation*}
\tilde{\mathbf{T}}_{|a,b|}([x],[y])=\bigoplus_{x\in[x],y\in[y],|x-y|<1}\mathbf{T}'_{|a,b|}(x,y).
\end{equation*}
Note that $\Theta(\tilde{\mathbf{T}}_{|a,b|})=\mathbf{T}_{|a,b|}$.

%\begin{theorem}[Hanson and Rock \cite{2020Decomposition}]
   % Let $\tilde{\mathbf{A}}_{\mathbb{R}}$ be a continuous quiver of type $\tilde{A}$. Every pointwise finite-dimensional representation $\mathbf{V}$ of $\tilde{\mathbf{A}}_{\mathbb{R}}$ is decomposed into the direct sum of modules $\tilde{\mathbf{T}}_{|a,b|}$.
%\end{theorem}   

Denoted by $\mathrm{Rep}'_{k}(\tilde{\mathbf{A}}_{\mathbb{R}})$ the category of representations of $\tilde{\mathbf{A}}_{\mathbb{R}}$ with following decomposition
\begin{equation*}
  \tilde{\mathbf{M}}=\bigoplus_{i\in I}\tilde{\mathbf{T}}_{|x_i,y_i|}.
\end{equation*}

Similarly to the  maximal rigid representations in $\mathrm{Rep}'_{k}(\mathbf{Q})$, we can define the concept of maximal rigid representations in $\mathrm{Rep}'_{k}(A_\mathbb{R})$.

Let $n$ be an integer greater than $1$ and $\alpha=(a_{i})_{i\in\mathbb{Z}}$ be a sequence such that $a_{0}=0$, $a_{i}<a_{i+1}$ and $a_{i+n}=\sigma(a_{i})$ for any $i\in\mathbb{Z}$.   
Let $\tilde{\mathbf{M}}$ be an object in $\mathrm{Rep}'_{k}(\tilde{\mathbf{A}}_{\mathbb{R}})$. If $\Theta(\tilde{M})$ is of type $\alpha$, the $\tilde{M}$ is called of type $(a_0,a_1,\ldots,a_{n-1})$.

\begin{comment}
         \begin{enumerate}
        \item[(a)] for any objects $\tilde{\mathbf{V}}\in Rep'_{k}(\tilde{\mathbf{A}}_{\mathbb{R}})$ and $\mathbf{V}\in Rep'_{k}(\mathbf{Q})$, so
        \begin{equation*}
        \begin{aligned}
            \tau^{+}(\tilde{\mathbf{V}}([x]))&=\mathbf{V}(x),\\
            \tau^{+}(\tilde{\mathbf{V}}([x],[y]))&=\mathbf{V}(x,y);
        \end{aligned}   
        \end{equation*}
        \item[(b)] for any objects $\mathbf{V}\in Rep'_{k}(\mathbf{Q})$ and $\tilde{\mathbf{V}}\in Rep'_{k}(\tilde{\mathbf{A}}_{\mathbb{R}})$, so
        \begin{equation*}
        \begin{aligned}
            \tau^{-}(\mathbf{V}(x))&=\tilde{\mathbf{V}}([x]),\\
            \tau^{-}(\mathbf{V}(x,y))&=\tilde{\mathbf{V}}([x],[y]);
        \end{aligned}   
        \end{equation*}
    \end{enumerate}

\end{comment}

%Create an equivalence between the representation categories $Rep'_{k}(\tilde{\mathbf{A}}_{\mathbb{R}})$ of $\tilde{\mathbf{A}}_{\mathbb{R}}$ and the representation categories $Rep'_{k}(\mathbf{Q})$ of $\mathbf{Q}$ 

%Then, we give the following Lemma by Lemma \ref{eq}.

%\begin{lemma}
%\label{representations equal}
  %  The number of isomorphism classes of maximal rigid representations of $\tilde{\mathbf{A}}_{\mathbb{R}}$ is equal to that of $\mathbf{Q}$.
%\end{lemma}

As a corollary of Theorem \ref{main result}, we have the following theorem.

\begin{theorem}
    Let $\tilde{\mathbf{A}}_{\mathbb{R}}$ be a continuous quiver of type $\tilde{A}$ and $\alpha=(a_{0},a_{1},a_{2},\cdots,a_{n-1})$ with $a_{0}=0$ and $a_{i}<a_{i+1}\in [0,1)$ for any $0\leq i\leq n-2$. Let 
\begin{equation*}
     \tilde{\mathbf{A}}=\{[\tilde{\mathbf{M}}]|\tilde{\mathbf{M}} \textrm{ is a maximal rigid representation of  $\tilde{A}_{\mathbb{R}}$ of type $\alpha$}\}
\end{equation*}
Then
\begin{equation*}
        |\tilde{A}|=2^{n+1}\left(
        \begin{aligned}
            4n-1\\
            2n-1
        \end{aligned}
        \right).
    \end{equation*}
\end{theorem}

\bibliographystyle{abbrv}
\bibliography{ref2}

\end{document}